\documentclass[12pt]{amsart}

\usepackage{amssymb}
\usepackage{graphicx}
\usepackage{hyperref}
\usepackage[hyphenbreaks]{breakurl}
\usepackage{color}

\numberwithin{equation}{section}
\numberwithin{figure}{section}

\allowdisplaybreaks
\DeclareSymbolFont{bbold}{U}{bbold}{m}{n}
\DeclareSymbolFontAlphabet{\mathbbold}{bbold}

\newcommand{\argmin}{\mathrm{argmin}}

\newcommand{\cA}{\mathcal{A}}
\newcommand{\cB}{\mathcal{B}}
\newcommand{\cC}{\mathcal{C}}

\newcommand{\bE}{\mathbb{E}}
\newcommand{\bP}{\mathbb{P}}
\newcommand{\bR}{\mathbb{R}}

\newtheorem{theorem}{Theorem}[section]

\newtheorem{corollary}[theorem]{Corollary}
\newtheorem{proposition}[theorem]{Proposition}
\newtheorem{remark}[theorem]{Remark}

\begin{document}

\title[Fortune favors the bold]{Leading the field: Fortune favors the bold \\in Thurstonian choice models}

\title[Fortune favors the bold]{Leading the field: Fortune favors the bold \\in Thurstonian choice models}

\author[S.N. Evans]{Steven N. Evans}
\address{Department of Statistics \#3860\\
 367 Evans Hall \\
 University of California \\
  Berkeley, CA  94720-3860 \\
   USA} 
\email{evans@stat.berkeley.edu}

\author[R.L. Rivest]{Ronald L. Rivest}
\address{Computer Science and Artificial Intelligence Lab \\
Massachusetts Institute of Technology \\
	Cambridge, MA 02139 \\
	USA} 
\email{rivest@mit.edu}

\author[P.B. Stark]{Philip B. Stark}
\address{Department of Statistics \#3860\\
 367 Evans Hall \\
 University of California \\
  Berkeley, CA  94720-3860 \\
   USA} 
\email{evans@stat.berkeley.edu}

\thanks{SNE supported in part by NSF grants DMS-09-07630 and DMS-15-12933 and NIH grant 1R01GM109454, 
RLR supported in part by NSF Science \& Technology Center grant CCF-0939370
We thank Alex Rivest for suggesting the procedure given in Appendix C for correcting small-school bias.}



\date{\today}

\maketitle

\begin{abstract}
    Schools with the highest average student performance are often the smallest schools;
    localities with the highest rates of some cancers are frequently small;
    and the effects observed in clinical trials are likely to be largest for the 
    smallest numbers of subjects.
    Informal explanations of this ``small-schools phenomenon'' point to the fact that the 
    sample means of smaller samples have higher variances.
    But this cannot be a complete explanation:
    If we draw two samples from a diffuse distribution that is symmetric about some point, 
    then the chance that the smaller sample has larger mean is 50\%.
    A particular consequence of results proved below is that if one draws three or more  
    samples of different sizes from the same normal distribution,
    then the sample mean of the smallest sample is most likely to be highest, 
    the sample mean of the second smallest sample is second most likely to be highest, 
    and so on; this is true even though for any pair
    of samples, each one of the pair is equally likely to have the larger sample mean.
    The same effect explains why heteroscedasticity  
    can result in misleadingly small nominal $p$-values 
    in nonparametric tests of association.       
  
    Our conclusions are relevant to certain stochastic choice models,
    including the following generalization of Thurstone's Law of Comparative Judgment.  
    There are $n$ items.
    Item $i$ is preferred to  item $j$ if $Z_i < Z_j$, 
    where $Z$ is a random $n$-vector of preference scores.
    Suppose $\bP\{Z_i = Z_j\} = 0$ for $i \ne j$, so there are no ties. 
    Item $k$ is the favorite if $Z_k < \min_{i\ne k} Z_i$.  
    Let $p_i$ denote the chance that item $i$ is the favorite.  
    We characterize a large class of distributions for $Z$  for which $p_1 > p_2 > \cdots > p_n$.    
    Our results are most surprising when  $\bP\{Z_i < Z_j\} = \bP\{Z_i > Z_j\} = \frac{1}{2}$ for  $i \ne j$, 
    so neither of any two items is likely to be preferred over the other in a pairwise 
    comparison.
    Then, under suitable assumptions, $p_1 > p_2 > \cdots > p_n$ when the 
    variability of $Z_i$ decreases with $i$ in an appropriate sense.  
    Our conclusions echo the proverb  ``Fortune favors the bold.''
    
\end{abstract}

\section{Introduction}

When an achievement test is administered to all students of a particular age in a U.S. state,
it is typically observed that there is a disproportionate number of small schools among those
with the highest average scores \cite{wainer07}.  This ``small-schools phenomenon'' is to
be expected even if the scores of individual students at small schools come from the
same population as those at other schools: the standard deviation of the
average score at a school with $n$ students is proportional to $\frac{1}{\sqrt{n}}$;
averages at small schools will thus be more variable than those at larger schools; and hence
small schools are likely to be disproportionately represented among the highest performing
(and lowest performing) schools.  \cite{wainer07} lists several more examples
of this effect such as small communities having what seem to be unusually high rates of
kidney cancer and small cities appearing to be safer than larger ones.

The results we establish here bear on how the probability that a school has the
highest average depends on its size under the assumption that student performances
are drawn from a common population.  
Suppose that the $n \ge 3$ schools are listed in
order of increasing size and $Z_1, Z_2, \ldots, Z_n$ are the average test scores.
Assume that the $Z_k$ are independent and symmetrically distributed around a
common mean $\mu$ and that $|Z_i - \mu|$ stochastically dominates $|Z_j - \mu|$
when $i < j$ (for example, this will hold approximately when each $Z_k$
is approximately Gaussian because then $Z_k$ is approximately of the form
$\frac{\sigma}{\sqrt{N_k}} Y_k$, where $\sigma$ is the standard deviation for
an individual student's score, $N_k$ is the number of students at the $k^{\mathrm{th}}$
school, and $Y_1, Y_2, \ldots, Y_n$ are independent standard normal random variables).
It follows from the results we establish that
$
\bP\{\text{$Z_k$ is the largest of $Z_1, \ldots, Z_n$}\}
$
is decreasing in $k$---the smaller a school is the more likely it is
to have the highest average test score---even though 
no school has an advantage over any other in a ``head-to-head'' competition
($\bP\{Z_i > Z_j\} = \bP\{Z_j > Z_i\}$ for any pair $i \ne j$).

We can describe this conclusion a little more picturesquely.
Consider a group of $n \ge 3$ independent individuals with equal skill, in the sense that 
each individual's performance is symmetrically distributed about some common mean,
so that in a head-to-head contest between any two there is an equal 
chance that either will win.
For each individual $k$, let $\Pr_k(x)$ be 
the chance that the absolute value of
the difference between his or her performance and the 
shared expected performance exceeds $x$.
Suppose the individuals are well ordered by these probabilities: for all $x > 0$,
$\Pr_1(x) > \Pr_2(x) > \cdots > \Pr_n(x)$.
Under these assumptions, individual~1 has the highest probability of having
the best performance, individual~2 has the second-highest, and so on.
If a greater chance of extreme performance results from deliberate risk-taking, then
individual~1 is the ``boldest'' and the most likely to perform best.
In this sense, fortune favors the bold.
(Of course, symmetry dictates that individual~1 is also most likely to perform worst!)

To make our results mathematically precise and to connect them to the
literature on stochastic models for ranking and ordering, we require the
following notation and terminology.
Label the $n$ items with the set of integers $[n] := \{1,2,\ldots,n\}$.
An individual's preferences can be represented in two related ways:
either we report the {\em order vector} $(w(1), w(2), \ldots, w(n))$,
where $w(1) \in [n]$ is the label of the most favored item, $w(2)$
is the label of the second-most favored item, etc., or we report
the corresponding {\em rank vector} $(y(1), y(2), \ldots, y(n))$,
where $w(y(i)) = i$ for $i \in [n]$.

The order vector and rank vector are permutations
of the set $[n]$.
There is a huge literature on models of
random permutations, much of which attempts to capture
features of how individuals actually go about assigning orders or rankings
using whatever information they have at their disposal.
The standard reference
is \cite{MR1346107}, with \cite{MR964069,MR1237197} as useful adjuncts.

The earliest model for assigning orders is due to 
\cite{Thu27,Thu31a}.  
In Thurstone's model, the item labeled $i$ is associated
with a (real-valued) random variable $Z_i$, where the random vector
$Z = (Z_1,  \ldots, Z_n)$ is such that $\bP\{Z_i = Z_j\} = 0$
for $i \ne j$, and the resulting order vector is $(i_1, i_2, \ldots, i_n)$
if $Z_{i_1} < Z_{i_2} < \ldots < Z_{i_n}$.
One may interpret $-Z_i$
as the desirability of item $i$ measured on a one-dimensional
scale, so that items are ordered in decreasing order of their
desirability.

In many applications, it is more natural to consider $Z_i$ rather
than~$-Z_i$ to be the desirability.  
For example, one might model an election 
by letting~$Z_i$ be the number of voters who
will vote for candidate~$i$ (see, for example,
\cite{Laslier06}).  
The candidate who garners the \emph{most}
votes wins the election.  
As another example, consider 
\emph{Thompson Sampling} for the multi-armed
bandit problem in machine learning.  The random variable~$Z_i$ models the benefit
resulting from pulling arm~$i$.  The $Z_i$ are sampled and the arm
with \emph{maximum}~$Z_i$ is pulled
\cite{Thompson33,AgrawalGoyal11}.  Nonetheless, we shall continue
here to follow the tradition of Thurstone, and let $-Z_i$ model the
desirability of item~$i$.

Let $Z_*$ denote $\min\{Z_i : 1 \le i \le n\}$, the smallest $Z_i$ value, and
let $I_*$ denote $\argmin \{Z_i : 1 \le i \le n\}$,
the index of the minimum $Z_i$ value.
Let $p_i$ denote the probability that the rank of item $i$ is $1$ (i.e., that
$i=I_*$).

Given a specification of $Z$,
there are three closely-related problems to consider:
\begin{enumerate}
\item Finding the distribution of $Z_*$.
See \cite{Gumbel58,MR1892574,MR2234156,MR2364939} for a sample of the
extensive body of work in this area.
\item Determining which $i$ is most likely to be $I_*$.
\item Given $i$, deriving an expression for $p_i$.
\end{enumerate}
We do not consider problem~1 here; our focus is on problem~2,
particularly when, as is usually the case,
solving~2 by solving~3 seems intractable.
Most generally, we are concerned with
finding conditions that imply that $p_1 > p_2 > \cdots > p_n$.

We briefly survey various distributional assumptions
on $Z = (Z_1, \ldots, Z_n)$ that have been considered
in this context.

Thurstone proposed taking $Z = (Z_1, \ldots, Z_n)$
to have a non-degenerate Gaussian
distribution.  Despite its conceptual simplicity, this model is rather
daunting computationally.
Here $p_i$ is the probability that $Z$ falls in the region
$R := \{z \in \bR^n : z_i < z_j, \, j \ne i\}$,
an intersection of half-spaces
$\bigcap_{j \ne i} \{z \in \bR^n : z_i - z_j < 0\}$.
We can write $Z = \mu + X \Sigma^{\frac{1}{2}}$, where $\mu$ is
the mean vector of $Z$, $\Sigma$ is the variance-covariance matrix of $Z$,
$\Sigma^{\frac{1}{2}}$ is the positive definite symmetric
square root of $\Sigma$, and $X$ is a vector with independent standard normal
entries.
We are therefore interested in the probability that
$X$ falls in the polyhedral region $(R - \mu) \Sigma^{-\frac{1}{2}}$.
It is usually not possible to express such probabilities in a simple closed
form, but  there is a large literature on approximating them numerically
using various ingenious recursive schemes---see, for example,
\cite{MR1944268,MR1959823,MR2412640}.

Appendix~A shows that when $\{Z_i\}_{i=1}^n$
are independent Gaussian random variables,
computing the probabilities
$\{p_i\}_{i=1}^n$ explicitly is somewhat complex even when $n = 3$.
Appendix~B shows that this is also true when
$\{Z_i\}_{i=1}^n$  is a vector of independent random variables
with bilateral exponential distributions.
However, if all one cares about
is the the ordering of the $p_i$'s, then
the results of the present paper may apply to cases where
explicitly computing $\{p_i\}$ is intractable.

\cite{MR0040629} suggested taking the random vector $Z$ in
Thurstone's general model to be of the form
$(\theta_1 + X_1, \ldots, \theta_n + X_n)$,
where $\theta_1, \ldots, \theta_n$ are real-valued parameters and
$X_1, \ldots, X_n$
are independent and identically distributed (IID) random variables.
Equivalently (by exponentiating), one can take $Z$ to be of the form
$(\gamma_1 Y_1, \ldots, \gamma_n Y_n)$, where $\gamma_1, \ldots, \gamma_n$
are positive parameters and $Y_1, \ldots, Y_n$ are IID
positive random variables.  It is a consequence of our results here
that if $\theta_i < \theta_j$ (or $\gamma_i < \gamma_j$), then
$i$ is at least as likely as $j$ to have rank $1$, and this inequality is strict
under mild conditions.  
\cite{MR0093876}
provides a number of other results about the dependence on the
parameters of various other probabilities related to the order
and rank vectors.

A particularly tractable example of the multiplicative version of
Daniels' type of Thurstonian model is when
$(Z_1, \ldots, Z_n) = (\gamma_1 Y_1, \ldots, \gamma_n Y_n)$
with $Y_1, \ldots, Y_n$ IID
exponential random variables.  In this case the probability of
a given order vector $(i_1, \ldots, i_n)$ can be computed explicitly: it is
\[
\frac{\lambda_{i_1}}{\sum_j \lambda_j}
\frac{\lambda_{i_2}}{\sum_{j \ne i_1} \lambda_j}
\frac{\lambda_{i_3}}{\sum_{j \ne i_1, i_2} \lambda_j}
\cdots,
\]
where $\lambda_i := \gamma_i^{-1}$ for $1 \le i \le n$.
This model is due to \cite{MR0391338} and \cite{MR0108411},
and was studied in \cite{MR2630352,Sil84} as the {\em vase} model:
if we imagine a vase containing $n$ types of balls with balls of
type $i$ being in proportion $\lambda_i/(\sum_j \lambda_j)$ and we
remove balls one-by-one uniformly without replacement, then the
order in which the $n$ types first appear is given by this model.
The Plackett and Luce model is the only Thurstonian model of the Daniels type
that satisfies the axioms laid out in \cite{MR0108411}
for a rational choice procedure---see \cite{Yel77} for a discussion.

The Plackett and Luce model is also the
stationary distribution of a discrete-time Markov chain that is sometimes
called the {\em Tsetlin library process} or the {\em move-to-the-front
self-organizing list}.
Here the items are pictured as
books and an order vector $(i_1, \ldots, i_n)$ corresponds to a stack
with the book labeled $i_n$ on the bottom and the book labeled $i_1$
on top.
In each step of the chain, book $i$ is chosen with probability
proportional to $\lambda_i$, removed from its current position in the stack,
and placed on top of the stack.
See, e.g., \cite{Riv76a}
for early work on this process, and \cite{MR1371073}
for a detailed analysis of this Markov chain and an extensive review
of the related literature.

Thurstonian models based on random vectors with much more complex structure
are discussed in \cite{MR1237206,MR2312235}.

Subsection~\ref{SS:randomized_experiments} presents
a third, more involved, example that illustrates
a model of a more complex type
that is not built from IID random variables, 
but where the assumptions of our main result, Theorem~\ref{T:main}, 
giving the ordering of $\{p_i\}$, still applies.  
This example is cast in terms of the
times taken by three workers to complete three randomly assigned
tasks.  
The expected time for a worker to complete a task
is the same for every (worker, task) pair, but the performance of the
first worker is more variable than that of the second, which is in turn more variable than that of the third.
Again, computing $\{p_i\}$
is tedious and complex, but Theorem~\ref{T:main} easily allows
one to find their ordering without explicit computation
and to conclude that the first worker has the highest probability of finishing
first and the second worker has the second highest probability
of finishing first.

\emph{This paper investigates how to determine,
in Thurstonian models, the ordering
of the probabilities that each of the given
items will be the most preferred,
without having to explicitly compute these probabilities.}

In other words, we study the distribution of the first entry in the order
vector or, equivalently, the distribution of the label of the item
with rank one,  and we seek conditions on the
distribution of the random vector $(Z_1, \ldots, Z_n)$ such that
if $p_i$ is the probability that the item labeled $i$ has rank one,
then $p_1 > p_2 > \cdots > p_n$ or at least $p_1 \ge p_2 \ge \cdots \ge p_n$.  
As we have already remarked, we show that the
chain of weak inequalities holds in the Daniels model if
$\theta_1 < \theta_2 < \ldots < \theta_n$ in the additive case
and $\gamma_1 < \gamma_2 < \ldots < \gamma_n$ in the multiplicative case.

The strict inequalities also hold under suitable assumptions.
To see that extra assumptions are necessary, suppose we are in the additive case with $n=3$ and
the common distribution of $X_1, X_2, X_3$ is uniform on the interval $[0,1]$,
with
$\theta_1 = 0 < \theta_2 = 1 < \theta_3 = 2$.
Then $p_1 = 1 > p_2 = 0 = p_3$,
so only weak and not strict inequalities hold in general.
The conclusion $p_1 > p_2 > \cdots > p_n$ can be verified by direct
computation for the Plackett and Luce model, where
$p_i = \lambda_i / (\lambda_1 + \cdots + \lambda_n)$ with
$\lambda_i = \gamma_i^{-1}$.

The plan of the remainder of the paper is as follows.
In Section~\ref{S:Gaussian_example} we consider a Thurstonian
model with $(Z_1, \ldots, Z_n) = (\sigma_1 X_1, \ldots, \sigma_n
X_n)$, where the $\sigma_i$ are positive constants and $(X_1, \ldots,
X_n)$ is a random vector with IID standard Gaussian entries.
Of course, if $n=2$, then $p_1 = p_2 = \frac{1}{2}$ by the symmetry
of the Gaussian distribution, but we show
in Section~\ref{S:Gaussian_example} that
if $n \ge 3$ and $\sigma_1 >\sigma_2 >
\ldots > \sigma_n$, then $p_1 > p_2 > \cdots > p_n$.
In Appendix~A we compute $\{p_1, p_2, p_3\}$ for $n=3$
to emphasize the difficulty of establishing
by direct computation that such an ordering holds for general $n$.

One way to think about this result is that a choice is being made
among $n$ individuals based on their responses to a set of stimuli.
The IID random variables $\{|X_1|, \ldots, |X_n| \}$ represent the random
stimuli given to the individuals.  The response of individual $k$
to the stimulus $|X_k|$ is $S_k \eta_k(|X_k|)$, where
$\eta_k(y) = \sigma_k y$ and $S_k$ is the sign of $X_k$, a $\{-1,+1\}$-valued
random variable that is independent of $|X_k|$ and equally likely to be $-1$
or $+1$.  
For each $k$, the function $\eta_k$ happens to be increasing---but
as we shall see, that is irrelevant for a conclusion like that above.
What is important is that $\eta_i(y) > \eta_j(y)$ for all $y$
and $1 \le i < j \le n$,
so that if individuals $i$ and $j$ receive the same stimulus, the response of 
individual $i$ will be more extreme than
that of individual $j$.
The expected responses $\bE[Z_k]$, $1 \le k \le n$, are all zero
and $\bP\{Z_i > Z_j\} = \bP\{Z_i < Z_j\} = \frac{1}{2}$,
$1 \le i \ne j \le n$, so that individual $i$ has no advantage over
individual $j$ in a head-to-head contest, and yet $p_1 > p_2 > \cdots > p_n$.

These observations suggest that a similar result might hold if
\[
    (Z_1, \ldots, Z_n) = (S_1 \eta_1(Y_1), \ldots, S_n \eta_n(Y_n))\ ,
\]
where $(S_1, \ldots, S_n)$
is a suitable exchangeable $\{-1,+1\}^n$-valued random vector (recall that
a random vector is exchangeable if its joint distribution
is unchanged by any permutation of the coordinates),
$(Y_1, \ldots, Y_n)$ is an exchangeable $E^n$-valued
random vector for some measurable space $E$, and the functions
$\eta_k: E \to
\bR_+$ have the property that $\eta_i(y) > \eta_j(y)$ for all
$y \in E$ and
$1 \le i < j \le n$ (so that the response $Z_i$ is ``bolder'' than
the response $Z_j$).
We show in Section~\ref{S:main} that this conclusion is indeed
valid under appropriate assumptions (e.g., the
ordering of the $p_k$ would not hold if $S_k = +1$ with probability one
for all $k$; to rule this sort of situation out, we require
\[
\bP\{\#\{k \in [n] : S_k = -1\} = 2\}
\ge
\binom{n}{2} \bP\{\#\{k \in [n] : S_k = -1\} = 0\},
\]
which holds, for example, when $\{S_k\}$ are IID with individual
probability at least $\frac{1}{2}$ of taking the value $-1$).

In Section~\ref{S:independent} we look at the special case
in which $\{Y_1, \ldots, Y_n \}$ and $\{S_1, \ldots, S_n\}$ are both
IID.

We give two applications of our results in Section~\ref{S:applications}.
In Subsection~\ref{SS:randomized_experiments} we consider a model for randomized experiments
where  $n$ treatments are assigned uniformly at random to $n$ individuals.  The distribution of the
response of individual $j$ to treatment $i$ is symmetrically distributed about zero.
For a fixed individual $j$ the distribution of the
magnitude of the effect of treatment $i$ is stochastically nonincreasing in $i$:
Lower numbered treatments are more likely to have larger magnitude effects than
higher numbered ones.  
We will show that treatment $1$ is most likely to have
the greatest effect, treatment $2$ is second most likely to have
the greatest effect, and so on, even though no treatment
causes any systematic benefit or harm to any individual.  

In Subsection~\ref{SS:nonparametric_association} we use our results to show
that heteroscedasticity can distort the $p$-value of a permutation-based
test for association between two series to make it appear that there is
positive or negative association between the two series when there is no
such systematic relationship.

Appendix C sketches an approach for removing the ``small-school bias'' in a
way that is both \emph{fair} (equally likely to choose as best any school, when
the schools have the same effect on student scores) and \emph{valid} (most likely
to choose as best the school that increases student scores the most).

\section{Motivating Gaussian example}
\label{S:Gaussian_example}

Our interest in the general topic of this paper was piqued by the
following observation about a Gaussian version of
the Thurstone model we mentioned in the Introduction.

\begin{proposition}
\label{P:Gaussian}
Suppose $n \ge 3$ and
$(Z_1, \ldots, Z_n) =(\sigma_1 X_1, \ldots, \sigma_n X_n)$,
where $\sigma_i > 0$
for $1 \le i \le n$ and the entries of
the random vector $(X_1, \ldots, X_n)$ are independent standard Gaussian
random variables.
If $\sigma_1 > \sigma_2 > \cdots > \sigma_n$,
then $p_1 > p_2 > \cdots > p_n$.
\end{proposition}

\begin{proof}
Let $\bigwedge \{\cdot \}$  denote the minimum of a set of real numbers and
$\bigvee \{\cdot\}$ denote the maximum.
Note that
\begin{eqnarray}
p_i & = & \bP\{\sigma_i X_i < \sigma_k X_k, \, k \ne i\}  \nonumber \\
& = &\bP\left\{\sigma_i X_i < \bigwedge_{k \ne i} \sigma_k X_k\right\} \nonumber \\
& = &\bP\left\{\bigvee_{k \ne i} (\sigma_i X_i - \sigma_k X_k) < 0\right\},
\end{eqnarray}
for $1 \le i \le n$.

Let $\phi$ and $\Phi$ denote the standard Gaussian probability density
function and cumulative distribution function, respectively.
Then (by conditioning on $X_i$ in the first integral, integrating by parts in the second,
and applying the chain rule in the third),
\begin{eqnarray}
p_i & = &
  \int_{-\infty}^\infty
      \prod_{j \ne i}
          \left(1 - \Phi\left(\frac{x}{\sigma_j}\right)\right)
          \frac{\partial}{\partial x} \Phi\left(\frac{x}{\sigma_i}\right) \, dx
\nonumber \\
& = &
- \int_{-\infty}^\infty
\Phi\left(\frac{x}{\sigma_i}\right)
\frac{\partial}{\partial x} \prod_{j \ne i} \left(1 - \Phi\left(\frac{x}{\sigma_j}\right)\right) \, dx \nonumber \\
& = &
- \int_{-\infty}^\infty
\Phi\left(\frac{x}{\sigma_i}\right)
\sum_{j \ne i} \phi\left(\frac{x}{\sigma_j}\right) \frac{1}{\sigma_j}
 \prod_{k \ne i,j} \left(1 - \Phi\left(\frac{x}{\sigma_k}\right)\right) \, dx, 
\end{eqnarray}
and so
\begin{eqnarray}
\frac{\partial p_i}{\partial \sigma_i}
& = &
-\sum_{j \ne i}
\int_{-\infty}^\infty
\phi\left(\frac{x}{\sigma_i}\right)\left(-\frac{x}{\sigma_i^2}\right)
 \phi\left(\frac{x}{\sigma_j}\right) \frac{1}{\sigma_j}
 \prod_{k \ne i,j} \left(1 - \Phi\left(\frac{x}{\sigma_k}\right)\right)
 \, dx \nonumber \\
& = &
\sum_{j \ne i}
\int_0^\infty
\phi\left(\frac{x}{\sigma_i}\right)\left(\frac{x}{\sigma_i^2}\right)
 \phi\left(\frac{x}{\sigma_j}\right) \frac{1}{\sigma_j} \nonumber \\
 & &\quad \times
 \left[
 \prod_{k \ne i,j} \left(1 - \Phi\left(\frac{x}{\sigma_k}\right)\right)
 -
 \prod_{k \ne i,j} \left(1 - \Phi\left(\frac{-x}{\sigma_k}\right)\right)
 \right]
 \, dx \nonumber \\
& >& 0,
\end{eqnarray}
where we used the facts that $\phi(z) = \phi(-z)$ for all $z \in \bR$
and that the function $\Phi$ is increasing.  It follows that
$p_i$ is an increasing function of $\sigma_i$, and,
because $p_i = p_j$ when $\sigma_i = \sigma_j$, it is clear that if
$\sigma_1 > \sigma_2 > \cdots > \sigma_n$, then
$p_1 > p_2 > \cdots > p_n$.
\end{proof}

\begin{remark}
We show in Appendix A that when $n=3$
\begin{eqnarray}
p_1
& = &
\frac{1}{2 \pi}
\arccos
\left(
-\frac{\sigma_1^2}{\sqrt{(\sigma_2^2 + \sigma_1^2) (\sigma_3^2 + \sigma_1^2)}}
\right) \nonumber \\
& > & p_2
= \frac{1}{2 \pi}
\arccos
\left(
-\frac{\sigma_2^2}{\sqrt{(\sigma_1^2 + \sigma_2^2) (\sigma_3^2 + \sigma_2^2)}}
\right) \nonumber \\
& > & p_3
= \frac{1}{2 \pi}
\arccos
\left(
-\frac{\sigma_3^2}{\sqrt{(\sigma_1^2 + \sigma_3^2) (\sigma_2^2 + \sigma_3^2)}}
\right), 
\end{eqnarray}
but finding such explicit expressions for the $p_i$ and establishing the
ordering claimed in Proposition~\ref{P:Gaussian} becomes increasingly
complex for larger values of $n$.  Moreover, Proposition~\ref{P:Gaussian}
holds, with essentially the same proof, if the common distribution of
$X_1, \ldots, X_n$ is an arbitrary symmetric distribution possessing a density,
whereas it is typically impossible to find explicit closed form expressions
for the $p_i$ in this case.  We observe in Appendix B that even for
a symmetric distribution as tractable as the bilateral exponential,
the formulae for the $p_i$ are already somewhat formidable for $n=3$
and establishing an ordering analogous to that
claimed in Proposition~\ref{P:Gaussian} requires a certain amount
of algebraic manipulation.
\end{remark}

\section{Main theorem}
\label{S:main}

This section presents our main theorem, giving the most general conditions
we have found so far that imply $p_1 \ge p_2 \ge \cdots \ge p_n$.


\begin{theorem}
\label{T:main}
Let $(Z_1,  \ldots, Z_n)$ be an $\bR^n$-valued
random vector given by $Z_k = S_k \eta_k(Y_k)$,
$1 \le k \le n$, where:
\begin{itemize}
\item
$(Y_1, \ldots, Y_n)$ is an exchangeable $E^n$-valued
random vector for some measurable space $(E,\mathcal{E})$;
\item
$\eta_1,  \ldots, \eta_n$ are measurable functions from
$E$ to $\bR_+$ with the property that
$\eta_i(y) \ge \eta_j(y)$ for all $y \in E$
and $1 \le i < j \le n$;
\item
$(S_1, \ldots, S_n)$ is an exchangeable
$\{-1,+1\}^n$-valued random vector;
\item
$(Y_1, \ldots, Y_n)$ and $(S_1, \ldots, S_n)$
are independent;
\item
$\bP\{S_1 = \cdots = S_{n-2} = +1; \; S_{n-1} = -1\}
\ge
\bP\{S_1 = \cdots = S_{n-1} = +1\}$.
\end{itemize}
Define
\[
p_k := \bP
\left \{
Z_k < \bigwedge_{\ell \ne k} Z_\ell
\right \}.
\]
Then, $p_1 \ge p_2 \ge \cdots \ge p_n$.
\end{theorem}

\begin{proof}
Let $(T_1, \ldots, T_n)$ be a vector of independent
random variables that is independent of the pair of random vectors
$(Y_1, \ldots, Y_n)$ and $(S_1, \ldots, S_n)$
and such that each random variable $T_k$ has an exponential
distribution with mean $1$.  Set
$Z_k^\epsilon = S_k (\eta_k(Y_k) + \epsilon T_k)$ for $1 \le k \le n$
and $\epsilon > 0$.  It is clear that
$p_k$ is the limit as $\epsilon \downarrow 0$ of
\[
p_k^\epsilon
:=
\bP
\left \{
Z_k^\epsilon < \bigwedge_{\ell \ne k} Z_\ell^\epsilon
\right \}
\]
for $1 \le k \le n$, so it suffices to show that
$p_1^\epsilon \ge p_2^\epsilon \ge \cdots \ge p_n^\epsilon$.

Set
\[
q(m) :=
\begin{cases}
\bP\{S_1 = \cdots = S_m = +1; \; S_{m+1} = -1\},
& \quad 0 \le m < n-1, \\
\bP\{S_1 = \cdots = S_{n-1} = +1\},
&  \quad m = n-1.
\end{cases}
\]
By the assumptions of the theorem, for $0 \le m < n-1$,
$$q(m) = \bP\{S_{k_1} = \cdots = S_{k_m} = +1; \; S_{k_{m+1}} = -1\}$$
for any subset $\{k_1, \ldots, k_{m+1}\} \subseteq [n]$
of cardinality $m+1$, and
$$q(n-1) = \bP\{S_{k_1} = \cdots = S_{k_{n-1}} = +1\}$$
for any subset $\{k_1, \ldots, k_{n-1}\} \subseteq [n]$
of cardinality $n-1$.  
Thus, $q(0) \ge q(1) \ge \cdots \ge q(n-2)$
and, by assumption, $q(n-2) \ge q(n-1)$.

Suppose that $a_1, \ldots, a_n \in \bR_+$ are distinct.
If $a_k \ne \bigwedge_\ell a_\ell$, then
\[
\begin{split}
& \bP
\left \{
S_k a_k < \bigwedge_{\ell \ne k} S_\ell a_\ell
\right \} \\
& \quad =
\bP
\left(
\left\{S_k = -1\right\}
\cap
\left
\{S_\ell = +1, \,
\mbox{ $\forall \ell \ne k$ such that $a_\ell > a_k$}
\right \}
\right); \\
\end{split}
\]
whereas if
$a_k = \bigwedge_\ell a_\ell$, then
\[
\bP
\left \{
S_k a_k < \bigwedge_{\ell \ne k} S_\ell a_\ell
\right \}
=
\bP
\left \{
S_\ell = +1, \, \ell \ne k
\right \}.
\]
In either case,
\[
\bP
\left \{
S_k a_k < \bigvee_{\ell \ne k} S_\ell a_\ell
\right \}
= q\left(\#\left\{1 \le \ell \le n : a_\ell > a_k\right\}\right).
\]

The values of $|Z_1^\epsilon|, \ldots, |Z_n^\epsilon|$
are almost surely distinct.  For $1 \le k \le n$ set
\[
M_k := \#\left\{1 \le \ell \le n : |Z_\ell^\epsilon| > |Z_k^\epsilon|\right\}.
\]
We must show that
\[
\bE
\left[
q\left( M_i \right)
\right]
\ge
\bE
\left[
q\left(M_j \right)
\right]
\]
for $1 \le i < j \le n$;
or, equivalently after summing by parts, that
\[
\begin{split}
& q(0) \bP\{M_i \ge 0\} + \sum_{m=0}^{n-2} [q(m+1)-q(m)] \bP\{M_i \ge m+1\} \\
& \quad \ge
q(0) \bP\{M_j \ge 0\} + \sum_{m=0}^{n-2} [q(m+1)-q(m)] \bP\{M_j \ge m+1\}. \\
\end{split}
\]
Since $\bP\{M_i \ge 0\} = \bP\{M_j \ge 0\} = 1$ and
$q(0) \ge q(1) \ge \cdots \ge q(n-1)$,
it suffices to show that 
\[
\bP\{M_i \ge m\}
\le
\bP\{M_j \ge m\}
\]
for $1 \le m \le n-1$.

Fix $1 \le i < j \le n$.  Note that
\[
\bP\{M_i \ge m\}
=
\bP
\left\{
\exists k_1, \ldots, k_m \ne i : |Z_{k_h}^\epsilon| > |Z_i^\epsilon|,
\, 1 \le h \le m
\right\}
\]
and $\bP\{M_j \ge m\}$ is given by a similar expression.
Define functions $\tilde \eta_k$, $1 \le k \le n$, by
$\tilde \eta_i = \eta_j$, $\tilde \eta_j = \eta_i$, and $\tilde \eta_k = \eta_k$,
$k \notin \{i,j\}$.  Observe that
\[
\begin{split}
&  \left\{
\exists k_1, \ldots, k_m \ne i : |Z_{k_h}^\epsilon| > |Z_i^\epsilon|,
\, 1 \le h \le m
\right\} \\
& \quad =
\left\{
\exists k_1, \ldots, k_m \ne i : 
\eta_{k_h}(Y_{k_h})+\epsilon T_{k_h} > \eta_i(Y_i)+\epsilon T_i,
\, 1 \le h \le m
\right\} \\
& \quad \subseteq
\left\{
\exists k_1, \ldots, k_m \ne i : 
\tilde \eta_{k_h}(Y_{k_h})+\epsilon T_{k_h} > \tilde \eta_i(Y_i)+\epsilon T_i,
\, 1 \le h \le m
\right\} \\
\end{split}
\]
because $\tilde \eta_i(y) = \eta_j(y) \le \eta_i(y)$ and 
$\tilde \eta _k(y) \ge \eta_k(y)$ for $k \ne i$ (with equality unless $k=j$).
Define random variables $\tilde Y_k$, $1 \le k \le n$, by
$\tilde Y_i = Y_j$, $\tilde Y_j = Y_i$, and $\tilde Y_k = Y_k$,
$k \notin \{i,j\}$.  Define $\tilde T_k$, $1 \le k \le n$, similarly.
By exchangeability, $(Y_1, \ldots, Y_n)$
and $(\tilde Y_1, \ldots, \tilde Y_n)$ have the same distribution.  Of course,
$(T_1, \ldots, T_n)$ and $(\tilde T_1, \ldots, \tilde T_n)$ have the same
distribution.  Therefore,
\[
\begin{split}
& \bP\left\{
\exists k_1, \ldots, k_m \ne i : 
\tilde \eta_{k_h}(Y_{k_h})+\epsilon T_{k_h} > \tilde \eta_i(Y_i)+\epsilon T_i,
\, 1 \le h \le m
\right\} \\
& \quad =
\bP\left\{
\exists k_1, \ldots, k_m \ne i : 
\tilde \eta_{k_h}(\tilde Y_{k_h})+\epsilon \tilde T_{k_h} > \tilde \eta_i(\tilde Y_i)+\epsilon \tilde T_i,
\, 1 \le h \le m
\right\}. \\
\end{split}
\]
Now
\[
\begin{split}
& \left\{
\exists k_1, \ldots, k_m \ne i : 
\tilde \eta_{k_h}(\tilde Y_{k_h})+\epsilon \tilde T_{k_h} > \tilde \eta_i(\tilde Y_i)+\epsilon \tilde T_i,
\, 1 \le h \le m
\right\} \\
& \quad =
\left\{
\exists k_1, \ldots, k_m \ne j : 
\eta_{k_h}(Y_{k_h})+\epsilon  T_{k_h} >  \eta_j( Y_j)+\epsilon  T_j,
\, 1 \le h \le m
\right\} \\
& \quad =
\{M_j \ge m\}. \\
\end{split}
\]
Putting the above together gives
$\bP\{M_i \ge m\} \le \bP\{M_j \ge m\}$ as required.

\end{proof}

\begin{remark}
Assume the hypotheses of Theorem~\ref{T:main}.  Note that
\begin{eqnarray}
& &\bP\{S_1 = \cdots = S_{n-2} = +1; \; S_{n-1} = -1\} \nonumber \\
& &\quad =
\bP\{S_1 = \cdots = S_{n-2} = +1; \; S_{n-1} = -1; \; S_n = +1\} \nonumber \\
& &\qquad +
\bP\{S_1 = \cdots = S_{n-2} = +1; \; S_{n-1} = -1; \; S_n = -1\} \nonumber \\
\end{eqnarray}
and
\begin{eqnarray}
& & \bP\{S_1 = \cdots = S_{n-1} = +1\} \nonumber \\
& &\quad =
\bP\{S_1 = \cdots = S_{n-1} = +1; \; S_n = +1\} \nonumber \\
& &\qquad +
\bP\{S_1 = \cdots = S_{n-1} = +1; \; S_n = -1\} \nonumber \\
& &\quad =
\bP\{S_1 = \cdots = S_{n-1} = +1; \; S_n = +1\} \nonumber \\
& &\qquad +
\bP\{S_1 = \cdots = S_{n-2} = +1; \; S_{n-1} = -1; \; S_n = +1\},
\end{eqnarray}
by the exchangeability hypothesis, so the hypothesis that
$$\bP\{S_1 = \cdots = S_{n-2} = +1; \; S_{n-1} = -1\}
\ge
\bP\{S_1 = \cdots = S_{n-1} = +1\}$$
is equivalent to the hypothesis that
$$\bP\{S_1 = \cdots = S_{n-2} = +1; \; S_{n-1} = S_n = -1\}
\ge
\bP\{S_1 = \cdots = S_n = +1\}. $$
Again using exchangeability, the latter is equivalent to
\[
\frac{1}{\binom{n}{2}} \bP\{\#\{k \in [n] : S_k = -1\} = 2\}
\ge
\bP\{\#\{k \in [n] : S_k = -1\} = 0\}.
\]
\end{remark}

\begin{remark}
Suppose in addition to the hypothesis of Theorem~\ref{T:main} that
$\bP\{S_i = S_j = +1\} = \bP\{S_i = S_j = -1\}$
for $i \ne j$.  Then, by exchangeability,
\begin{eqnarray}
\bP\{Z_i < Z_j\}
& = &
\bP\{\eta_i(Y_i) < \eta_j(Y_j)\} \bP\{S_i = S_j = +1\} \nonumber \\
& &\quad +
\bP\{\eta_i(Y_i) > \eta_j(Y_j)\} \bP\{S_i = S_j = -1\} \nonumber \\
& &\quad +
\bP\{S_i = -1; \; S_j = +1\} \nonumber \\
& = &
\bP\{\eta_i(Y_i) < \eta_j(Y_j)\} \bP\{S_i = S_j = -1\} \nonumber \\
& &\quad +
\bP\{\eta_i(Y_i) > \eta_j(Y_j)\} \bP\{S_i = S_j = +1\} \nonumber \\
& &\quad +
\bP\{S_i = +1; \; S_j = -1\} \nonumber \\
& = &
\bP\{Z_i > Z_j\}.
\end{eqnarray}
Theorem~\ref{T:main} is especially interesting in this case,
because then $Z_i$ is not systematically smaller than $Z_j$ for $i < j$,
and yet $p_1 \ge p_2 \ge \cdots \ge p_n$.
\end{remark}

\begin{remark}
\label{R:strict_inequality}
Theorem~\ref{T:main} gives a sufficient condition for 
the weak inequalities $p_1 \ge p_2 \ge \cdots \ge p_n$
but not the strict inequalities $p_1 > p_2 > \cdots > p_n$.
Examining the proof indicates how the hypotheses can be strengthened
to yield the latter conclusion.  Suppose that $\bP\{Z_i = Z_j\} = 0$ for $1 \le i \ne j \le n$.
It is clear from the proof of the theorem that $p_i > p_j$ for a given pair $1 \le i < j \le n$ if
and only if there exists $0 \le m \le n-2$ such that $q(m+1) < q(m)$ and
\[
\begin{split}
& \bP\left\{
\exists k_1, \ldots, k_{m+1} \ne i: 
\eta_{k_h}(Y_{k_h}) >  \eta_i( Y_i),
\, 1 \le h \le m+1
\right\} \\
& \quad <
\bP\left\{
\exists k_1, \ldots, k_{m+1} \ne j : 
\eta_{k_h}(Y_{k_h}) >  \eta_j( Y_j),
\, 1 \le h \le m+1
\right\} \\
\end{split}
\]
for that $m$.  For example,
if $n \ge 3$ and $q(0) = \bP\{S_1 = -1\} > \bP\{S_1 = +1; S_2 = -1\} = q(1)$, then it
suffices that
$
\bP\{\exists k \ne i : \eta_k(Y_k) > \eta_i(Y_i)\}
<
\bP\{\exists k \ne j : \eta_k(Y_k) > \eta_j(Y_j)\}
$
or, equivalently by exchangeability,
\[
\begin{split}
\bP\left\{\bigvee_{k \notin \{i,j\}} \eta_k(Y_k)  \vee \eta_j(Y_j) > \eta_i(Y_i) \right\}
& <
\bP\left\{\bigvee_{k \notin \{i,j\}} \eta_k(Y_k)  \vee \eta_i(Y_i) > \eta_j(Y_j) \right\} \\
& =
\bP\left\{\bigvee_{k \notin \{i,j\}} \eta_k(Y_k)  \vee \eta_i(Y_j) > \eta_j(Y_i) \right\}. \\
\end{split}
\]
Because $\eta_j(Y_j) \le \eta_i(Y_j)$ and $\eta_i(Y_i) \ge \eta_j(Y_i)$ it further suffices to have
\begin{equation}
\label{strict_inequality_condn_general}
0 < \bP\left\{\bigvee_{k \notin \{i,j\}} \eta_k(Y_k)  \vee \eta_i(Y_j) > \eta_j(Y_i), \; \bigvee_{k \notin \{i,j\}} \eta_k(Y_k)  \vee \eta_j(Y_j) \le \eta_i(Y_i) \right\}.
\end{equation}
\end{remark}

\section{Independent random variables}
\label{S:independent}

Theorem~\ref{T:main} has the following consequence
when the entries of $(Z_1, \ldots, Z_n)$ are independent.

\begin{corollary}
\label{C:independent}
Suppose that $n \ge 3$.
Let $(Z_1, \ldots, Z_n)$ be an $\bR^n$-valued
random vector given by $Z_k = S_k W_k$,
$1 \le k \le n$, where:
\begin{itemize}
\item
$W_1, \ldots, W_n$ are independent $\bR_+^n$-valued
random variables;
\item
$W_i$ stochastically dominates $W_j$ for $1 \le i < j \le n$
(that is, $\bP\{W_i > w\} \ge \bP\{W_j > w\}$
for all $w \in \bR_+$);
\item
$S_1, \ldots, S_n$ are IID
$\{-1,+1\}$-valued random variables with
$\bP\{S_k = +1\} \le \bP\{S_k = -1\}$;
\item
$(W_1, \ldots, W_n)$ and $(S_1, \ldots, S_n)$
are independent.
\end{itemize}
Define
\[
p_k := \bP
\left \{
Z_k < \bigwedge_{\ell \ne k} Z_\ell
\right \}.
\]
Then, $p_1 \ge p_2 \ge \cdots \ge p_n$.
\end{corollary}

\begin{proof}
It is possible to write $W_k = \eta_k(Y_k)$,
where $Y_1, \ldots, Y_n$ are IID random variables that each have the
uniform distribution on the interval $[0,1]$
and
\[
\eta_k(y)
:= \inf\{w \in \bR_+ : \bP\{W_k \le w\} \ge y\},
\quad y \in [0,1].
\]
It follows from the stochastic ordering assumption on
$W_1, \ldots, W_n$ that
$\eta_i(y) \ge \eta_j(y)$ for $y \in [0,1]$
and $1 \le i < j \le n$.

Also, if we write $p$ for the common value of
$\bP\{S_k = +1\}$, then
\begin{eqnarray}
& &\bP\{S_1 = \cdots = S_{n-2} = +1; \; S_{n-1} = -1\}
=
p^{n-2} (1-p) \nonumber \\
&& \quad \ge
p^{n-1}
=
\bP\{S_1 = \cdots = S_{n-1} = +1\}. 
\end{eqnarray}

The result now follows from Theorem~\ref{T:main}.
\end{proof}

\begin{remark}
\label{R:original_result}
A simple consequence of
Corollary~\ref{C:independent} is that
if $n \ge 3$, $V_1, \ldots, V_n$ are IID random variables
that are symmetrically
distributed (that is, the common distribution of
$V_k$ is the same as that of $-V_k$) and
$c_1 \ge c_2 \ge \cdots \ge c_n > 0$ are nonnegative constants,
then
\begin{equation}
\label{weighted_iid}
\bP
\left \{
c_i V_i < \bigwedge_{k \ne i} c_k V_k
\right \}
\ge
\bP
\left \{
c_j V_j < \bigwedge_{k \ne j} c_k V_k
\right \}
\end{equation}
for $1 \le i < j \le n$.
 
The discussion in Remark~\ref{R:strict_inequality} addresses when inequality in \eqref{weighted_iid} will be strict. Assume that $n \ge 3$ and $c_1 > c_2 > \cdots > c_n > 0$.
Writing $V_k = S_k |V_k|$, $1 \le k \le n$, where $(S_1, \ldots, S_n)$ is IID $\{-1,+1\}$-valued random variables that are independent
of $(|V_1|, \ldots, |V_n|)$ with $\bP\{S_k=\pm 1\} = \frac{1}{2}$,
we have 
$$\bP\{S_1 = -1\} = \frac{1}{2} > \frac{1}{4} = \bP\{S_1 = +1; S_2 = -1\}.$$
Suppose that
the common distribution of $V_k$, $1 \le k \le n$, is diffuse and that $0$ 
is in the support of this distribution.  
Then
\begin{equation}
\label{strict_inequality_condn_weighted_iid}
\bP\left\{\bigvee_{k \notin \{i,j\}} c_k|V_k| \vee c_i|V_j| > c_j|V_i|, \; \bigvee_{k \notin \{i,j\}} c_k|V_k|  \vee c_j|V_j| \le c_i|V_i| \right\} > 0
\end{equation}
for $1 \le i < j \le n$, which is the special case in the present setting of the sufficient condition \eqref{strict_inequality_condn_general} for strict inequality.
To see this, note first that for all $\epsilon > 0$ sufficiently small we have
\[
\begin{split}
& \bP\left\{c_i|V_j| > c_j|V_i|, \; c_j|V_j| \le c_i|V_i|, \; c_i |V_j| > c_j |V_j| > \epsilon \right\}  \\
& = \bP\left\{\frac{c_j}{c_i} < \frac{|V_j|}{|V_i|} \le \frac{c_i}{c_j}, \;  c_i |V_j| >  c_j |V_j| > \epsilon\right\} > 0 \\
\end{split}
\]
whereas
\[
\bP\left\{\bigvee_{k \notin \{i,j\}} c_k|V_k| \le \epsilon\right\} > 0
\]
for all $\epsilon > 0$.
In particular, we recover
Proposition~\ref{P:Gaussian}

It is worth noting that \eqref{weighted_iid} doesn't hold with a strict inequality under just the assumption that $V_1, \ldots, V_n$ are 
IID random variables with a diffuse, symmetric common distribution.
For example, assume that $n=3$ and $c_1 > c_2 > c_3 > 0$ are given.  Suppose that the common distribution of $|V_k|$, $1 \le k \le 3$,
 is supported on an interval $[a,b]$ where the
intervals $c_1[a,b]$, $c_2[a,b]$, $c_3[a,b]$ are pairwise disjoint.  Then
\[
\bP\{c_1 V_1 < c_2 V_2 \wedge c_2 V_3\} = \bP\{V_1 < 0\} = \frac{1}{2},
\]
\[
\bP\{c_2 V_2 < c_1 V_1 \wedge c_3 V_3\} = \bP\{V_1 > 0, \; V_2 < 0\} = \frac{1}{4},
\]
and
\[
\bP\{c_3 V_3 < c_1 V_1 \wedge c_2 V_2\} = \bP\{V_1 > 0, \; V_2 > 0\} = \frac{1}{4}.
\]
\end{remark}

\section{Applications}
\label{S:applications}

\subsection{Randomized experiments}
\label{SS:randomized_experiments}

Suppose we are interested in comparing $n$ treatments.  We will test
each treatment on one of $n$ individuals, which might be people,
families, banks, local or national economies, or plots of land, for
instance.  Treatments are assigned uniformly at random to individuals:
All $n!$ assignments are equally likely.  
The distribution of
the response of individual $j$ to treatment $i$ is a
distribution $P_{ij}$ that is symmetric about zero,
so that no treatment causes any systematic
benefit or harm to any individual.  Suppose for each fixed
$j \in [n]$ and all $y > 0$ that
$P_{ij}\{x \in \bR : |x| > y\}$ is nonincreasing in $i$,
so that the magnitude of the responses of a fixed individual to the
various treatments are stochastically nonincreasing in the treatment
number (i.e., low numbered treatments are more likely to have effects
with a large magnitude than high numbered treatments).
Suppose further that given the assignment of treatments
to individuals the responses of the individuals are conditionally
independent.

We can represent the response to treatment $i$ as
$Z_i = S_i \eta_i(\Pi_i, U_i)$,
where $S_1, \ldots, S_n$ are IID $\{-1,+1\}$-valued random variables
with $\bP\{S_i = -1\} = \bP\{S_i = +1\} = \frac{1}{2}$;
$(\Pi_1, \ldots, \Pi_n)$ is a uniform random permutation of $[n]$;
$U_1, \ldots, U_n$ are IID random variables with a uniform distribution
on the interval $[0,1]$; and $\eta_i(j,\cdot)$ is the inverse of the function
$y \mapsto P_{ij}\{x \in \bR : |x| > y\}$, that is,
\[
\eta_i(j,u)
:=
\sup\{y \ge 0: P_{ij}\{x \in \bR : |x| \le y\} < u\}.
\]
By assumption, $\eta_1(j,u) \ge \cdots \ge \eta_n(j,u)$, and it follows
from Theorem~\ref{T:main} that $p_1 \ge p_2 \ge \cdots \ge p_n$. Hence,
if we think of low values of the response as desirable, then low numbered
treatments are likely to appear to be the most desirable in a
single instance of the experiment, even though they
are also likely to appear to be the least desirable.

In order to give a simple, concrete example of this phenomenon,
consider a situation in which
there are three tasks of comparable difficulty that have to be completed
and three workers available to do them.  In terms of the setting above,
the tasks are the ``individuals'' and the workers are the ``treatments.''

Number the tasks $1$, $2$ and
$3$, and designate the workers by the letters $\cA$, $\cB$ and $\cC$.
The tasks are assigned to the workers at random, with the $3!=6$ possible
allocations being equally likely.
On average, the workers are equally
rapid at completing a given task, but the performance of Worker~$\cA$
is more variable than that of Worker $\cB$,
which is more variable than that of Worker $\cC$.

We model this very simply by assuming that the time
taken to perform Task~$1$ by Worker~$\cA$ (respectively,
Workers~$\cB$ and $\cC$) is either $T-A$ or $T+A$ (respectively,
$T-B$ or $T+B$, and $T-C$ or $T+C$) with equal probability,
where $A, B, C$ are positive constants.
Similarly, the respective times taken by the three workers to
perform Tasks~$2$ and $3$ are $T \pm a$, $T \pm b$, $T \pm c$
and $T \pm \alpha$, $T \pm \beta$, $T \pm \gamma$, with the two alternatives
in each case always being equally likely.
We assume that the
times taken by the workers are conditionally independent given
the random allocation of tasks (that is, all $2^3 = 8$ possible choices
of sign are equally likely for any particular allocation).

\begin{table}[h]
	\centering
\begin{tabular}{ r|c|c|c| }
 \multicolumn{1}{r}{}
  &  \multicolumn{1}{c}{Worker $\cA$}
  & \multicolumn{1}{c}{Worker $\cB$}
  & \multicolumn{1}{c}{Worker $\cC$}
  \\
 \cline{2-4}
 Task $1$ & $T \pm A$ & $T \pm B$ & $T \pm C$ \\
 \cline{2-4}
 Task $2$ & $T \pm a$ & $T \pm b$ & $T \pm c$ \\
 \cline{2-4}
 Task $3$ & $T \pm \alpha$ & $T \pm \beta$ & $T \pm \gamma$ \\
 \cline{2-4}
 \end{tabular}
\caption{Time for each of three workers to complete each of three tasks}
\label{tab:three_task_setup}
\end{table}

The relative variability of the workers' performance is modeled
by taking $A > B > C$, $a > b > c$,
and $\alpha > \beta > \gamma$.
The ordering among these nine quantities is otherwise arbitrary.
We thus have
an instance of the general situation considered above with the inconsequential
difference that the responses are symmetric about $T$ rather than $0$.
We will explore how the probability that a particular worker finishes first
depends on the ordering in detail.

Suppose the ordering is
$A > B > C > a > b > c > \alpha > \beta > \gamma > 0$.
Then worker~$\cA$ finishes first in the following scenarios:
\begin{enumerate}
    \item All signs are negative and $\cA$ is assigned task 1 (2 of 48)
    \item Only the first and second signs are negative and $\cA$ is assigned task~1, or $\cA$ is assigned task~2
             and $\cB$ is assigned task~3 (3 of 48)
    \item Only the first and third signs are negative and $\cA$ is assigned task~1, or $\cA$ is assigned task~2
             and $\cC$ is assigned task~3 (3 of 48)
    \item Only the first sign is negative (6 of 48)
    \item All signs are positive and $\cA$ is assigned task 3 (2 of 48)
\end{enumerate}
These comprise $16/48 = 1/3$ of the equally likely possibilities, so the chance that~$\cA$ finishes first
is $1/3$.
Similarly, worker~$\cB$ finishes first in the following scenarios:
\begin{enumerate}
    \item All signs are negative and $\cB$ is assigned task 1 (2 of 48)
    \item Only the first and second signs are negative and $\cB$ is assigned task~1, or
             $\cB$ is assigned task~2
             and $\cA$ is assigned task~3 (3 of 48)
    \item Only the second and third signs are negative and $\cB$ is assigned task~1, or $\cB$ is assigned task~2
             and $\cC$ is assigned task~3 (3 of 48)
    \item Only the third sign is negative (6 of 48)
    \item All signs are positive and $\cB$ is assigned task 3 (2 of 48)
\end{enumerate}
Again, these comprise $1/3$ of the possibilities, so the chance that~$\cB$ finishes first is $1/3$;
the same is true for~$\cC$.

However, if the ordering is
$A > a > \alpha > B > b > \beta > C > c > \gamma > 0$, then~$\cA$ finishes first if and only if
the first sign is negative, which has chance $1/2$.
For this ordering, $\cB$ finishes first if the first sign is positive and the second is negative,
which has chance $1/4$.
Worker $\cC$ finishes first if the first two signs are positive, which also has chance $1/4$.

It is possible to consider the various other possibilities that are
not the same as one of these two after a relabeling of the tasks; for example,
if $A > a > b > c > B > \alpha > \beta > \gamma > C  > 0$,
then the probability
that Worker $\cA$ finishes first is $\frac{5}{12}$, whereas the probabilities
that Workers $\cB$ and $\cC$ finish first are both $\frac{7}{24}$.
We do not present an exhaustive list of the results.

\subsection{Heteroscedasticity and nonparametric tests of association}
\label{SS:nonparametric_association}

The null hypothesis for standard nonparametric (permutation-based)
tests for association between two series,
such as the Spearman rank correlation test, amounts to the hypothesis
that one series is conditionally exchangeable given the other.
Heteroscedasticity can make that null hypothesis false, even when
there is no positive (resp. negative) association between the series,
where by positive (resp. negative) association we mean that, in some sense,
larger values of one variable tend to occur in conjunction with larger
(resp. smaller) values of the other.
Our results show qualitatively that this can distort the apparent $p$-value
of permutation tests for association.

Consider a decreasing deterministic sequence $x = (x_1, \ldots, x_n)$
and a sequence $Z = (Z_1, \ldots, Z_n)$ whose components are independent
and  symmetrically distributed, but such that $|Z_i|$ stochastically
dominates $|Z_j|$ for $1 \le i < j \le n$.
We haven't given a rigorous definition of association,
but $x$ and $Z$ are not associated in any
intuitively reasonable sense of the term.
However, Corollary~\ref{C:independent} shows that the first component
of $Z$ is most likely to be the largest; when that occurs,
the rank of the largest component of $Z$ is aligned
with the rank of the largest component of $x$.
The full distributional details are complicated, but
one might expect that an extension of this phenomenon
will tend to make the Spearman
rank correlation coefficient $r_S$ take more extreme values
than it would be
if the null hypothesis of exchangeability held.

The following simple example from \cite{walther97,walther99} shows that the
quantitative difference in probabilities can be quite striking.
Let $x = (4, 3, 2, 1)$ and
$$Z  = ( \sigma_1 Y_1, \sigma_2 Y_2, \sigma_3 Y_3, \sigma_4 Y_4),$$
where $\{Y_i\}$ are IID standard Gaussian variables,
$\sigma_1 = 2$, and $\sigma_2 = \sigma_3 = \sigma_4 = 1$.
The chance that $r_S = 1$ is the chance that $Z_1 > Z_2 > Z_3 > Z_4$.
If $\{Z_j\}$ were exchangeable, then that chance
would be $1/24 \approx 4.17\%$.
Simulation shows that in the heteroscedastic (non-exchangeable) model,
$$\bP\{r_S(X, Y)=1\}\approx 7\%,$$ about 68\% higher.
Calibrating the Spearman rank correlation test using the null hypothesis of
exchangeability is misleading, because heteroscedasticity alone
makes the components of $Z$ tend to be closer to ordered than they
would be under random permutations.

We can illustrate the phenomenon even more concretely with the
three workers and three tasks example
from Subsection~\ref{SS:randomized_experiments}.
Note that if $A > a > \alpha > B > b > \beta > C > c > \gamma > 0$, then the
distribution of the order in which the workers $\cA, \cB, \cC$ finish is uniform
over the four possibilities
$(\cA,\cB,\cC), \, (\cA,\cC,\cB), \, (\cB,\cC,\cA), \, (\cC,\cB,\cA)$
and the distribution of the Spearman rank correlation $r_S$
between the vector of finish
times for the three workers and the vector $(1,2,3)$ is
\[
\bP\{r_S = -1\} = \bP\left\{r_S = -\frac{1}{2}\right\}
= \bP\left\{r_S = +\frac{1}{2}\right\} = \bP\{r_S = +1\} = \frac{1}{4},
\]
whereas if the random vector of finish times were exchangeable
(that is, if we were in the usual null situation for the
Spearman rank correlation test), then
the distribution of $r_S$ would be
\[
\begin{split}
\bP\{r_S = -1\} & = \frac{1}{6}, \\
 \bP\left\{r_S = -\frac{1}{2}\right\}
& = \bP\left\{r_S = +\frac{1}{2}\right\} = \frac{1}{3}, \\
 \bP\{r_S = +1\} & = \frac{1}{6}, \\
\end{split}
\]
so performing a Spearman rank correlation test
would be likely to result
in the conclusion that there is a positive (or negative) association between
a worker's label and the worker's finish time.

Our results
do not predict the magnitude of the distortion of the
null distribution of $r_S$, but they do suggest
that there will be such a distortion
quite generally when one sequence is heteroscedastic with
an ordering of the degree of dispersion
that matches the ordering of magnitudes of the other,
even when the components of the first sequence are independent and
have equal means.

\section{Discussion and Conclusions}

We have presented general conditions
on a random vector 
\[
(Z_1, Z_2, \ldots, Z_n)
\]
that guarantee that the probabilities
$p_i := \bP\{Z_i < \bigwedge_{j \ne i} Z_j\}$ satisfy
$p_1 \ge p_2 \ge \cdots \ge p_n$; that is, that the probability
the $i^{\mathrm{th}}$ coordinate is the smallest is decreasing in
$i$.  Analogous results hold for the the probability that
the $i^{\mathrm{th}}$ coordinate is the largest.
The general conclusion is that ``Fortune favors the bold,'' and that
even if $\bP\{Z_i > Z_j\} =  \bP\{Z_i < Z_j\}$
for $1 \le i \ne j \le n$, so that no coordinate is systematically
larger than another, we can still have situations in which such
an ordering will occur because the variability of $Z_i$ decreases with $i$.
Our results give technical precision to the intuition embodied by
the proverb.
We emphasize that our results do not require the
explicit computation of the probability that $Z_i$ is extreme.

Presumably, even more general
conditions that determine the ranks of the probabilities that
each random variable will be extremal could be derived.
Similarly, we have considered inequalities
among the probabilities that different items will be most favored,
but it should also be possible to derive inequalities among the
probabilities that various subsets of the items will have various
subsets of the ranks, not just the chances that each individual item is best.
These remain open problems.

\bigskip
\noindent
{\bf Acknowledgment:} SNE supported in part by NSF grants DMS-09-07630 and DMS-15-12933 and NIH grant 1R01GM109454, 
RLR supported in part by NSF Science \& Technology Center grant CCF-0939370
We thank Alex Rivest for suggesting the procedure given in Appendix C for correcting small-school bias.

\bigskip

\providecommand{\bysame}{\leavevmode\hbox to3em{\hrulefill}\thinspace}
\providecommand{\MR}{\relax\ifhmode\unskip\space\fi MR }
\providecommand{\MRhref}[2]{%
  \href{http://www.ams.org/mathscinet-getitem?mr=#1}{#2}
}
\providecommand{\href}[2]{#2}

\section*{Appendix A. Three independent Gaussians}

Suppose that $X,Y,Z$ are independent zero mean Gaussian random vectors
with variances $\alpha^2 > \beta^2 > \gamma^2 > 0$.  Observe that
$\bP\{X < Y \wedge Z\} = \bP\{(Y-X,Z-X) \in Q\}$, where
$Q$ is the positive quadrant $\{(s,t) \in \bR^2 : s > 0, \, t > 0\}$.
The variance-covariance matrix of the pair $(Y-X,Z-X)$ is
\[
\Sigma
:=
\begin{pmatrix}
\beta^2 + \alpha^2 & \alpha^2 \\
\alpha^2 & \gamma^2 + \alpha^2
\end{pmatrix}.
\]
We can write
\[
(Y-X,Z-X) = (V,W) \Sigma^{\frac{1}{2}},
\]
where $\Sigma^{\frac{1}{2}}$
is the positive definite square root of the matrix $\Sigma$
and $(U,V)$ is a pair of independent standard Gaussian random variables.
The image of the quadrant $Q$ under the linear map defined by
$\Sigma^{-\frac{1}{2}}$ is a wedge with boundary given by the images
of the two positive coordinate axes.  Some algebra shows that
\begin{eqnarray}
&& \frac{
((1,0) \Sigma^{-\frac{1}{2}}) \cdot ((0,1) \Sigma^{-\frac{1}{2}})
}
{
\sqrt{((1,0) \Sigma^{-\frac{1}{2}}) \cdot ((1,0) \Sigma^{-\frac{1}{2}})}
\sqrt{((0,1) \Sigma^{-\frac{1}{2}}) \cdot ((0,1) \Sigma^{-\frac{1}{2}})}
} \nonumber \\
&& \quad =
-\frac{\alpha^2}{\sqrt{(\beta^2 + \alpha^2) (\gamma^2 + \alpha^2)}},
\end{eqnarray}
where we use $a \cdot b$ to denote the usual inner product
of two vectors $a$ and $b$.

It follows from the rotational symmetry of the distribution of $(U,V)$ that
\[
\bP\{X < Y \wedge Z\}
=
\frac{1}{2 \pi}
\arccos
\left(
-\frac{\alpha^2}{\sqrt{(\beta^2 + \alpha^2) (\gamma^2 + \alpha^2)}}
\right).
\]
A similar formula holds for
$\bP\{Y < X \wedge Z\}$ (resp. $\bP\{Z < X \wedge Y\}$)
by interchanging the roles of $\alpha^2$ and $\beta^2$ (resp. $\alpha^2$
and $\gamma^2$).

Some more algebra shows that
\[
\begin{split}
& \frac{\alpha^4}{(\beta^2 + \alpha^2) (\gamma^2 + \alpha^2)}
-
\frac{\beta^4}{(\alpha^2 + \beta^2) (\gamma^2 + \beta^2)} \\
& \quad =
\frac{(\alpha^2 - \beta^2) (\alpha^2 \beta^2 + \beta^2 \gamma^2 +  \alpha^2 \gamma^2)}
{(\alpha^2 + \beta^2) (\alpha^2 + \gamma^2) (\beta^2 + \gamma^2)}
> 0, \\
\end{split}
\]
and so
\[
\bP\{X < Y \wedge Z\} > \bP\{Y < X \wedge Z\}.
\]
Similarly,
\[
\bP\{Y < X \wedge Z\} > \bP\{Z < X \wedge Y\}.
\]

\section*{Appendix B. The minimum of three bilateral exponentials}

Given a dispersion parameter $\theta > 0$, write
$f_\theta(x) := \frac{1}{2 \theta} e^{-\frac{|x|}{\theta}}$ for the
density of the corresponding bilateral exponential distribution.  Note that
\[
\int_x^\infty f_\theta(t) \, dt
=
\begin{cases}
\frac{1}{2}(1 - e^{-\frac{|x|}{\theta}}) + \frac{1}{2},& \quad x < 0,\\
\frac{1}{2} e^{-\frac{|x|}{\theta}}, & \quad x \ge 0.\\
\end{cases}
\]
Suppose that $X,Y,Z$ are independent real-valued random variables
with respective bilateral exponential 
densities $f_a, f_b, f_c$, where the parameters satisfy
$a > b > c > 0$, so that $X$ is more dispersed than $Y$, which is
more dispersed than $Z$.

An explicit integration shows that
\[
\begin{split}
& \bP\{X < Y \wedge Z\} \\
& \quad = 
\int_{-\infty}^\infty
\bP\{Y > x\} \bP\{Z > x\}
\, \bP\{X \in dx\} \\
& \quad =
\frac{
2 a^3 b + a^2 b^2 + 2 a^3 c + 5 a^2 b c + 2 a b^2 c + a^2 c^2 +
 2 a b c^2 + b^2 c^2
}
{
4 (a + b) (a + c) (a b + b c + a c)
}. \\
\end{split}
\]
A similar expression for $\bP\{Y < X \wedge Z\}$
(resp. $\bP\{Z < X \wedge Y\}$) follows by interchanging
the roles of $a$ and $b$ (resp. $a$ and $c$).

It follows that
\[
\begin{split}
& \bP\{X < Y \wedge Z\} - \bP\{Y < X \wedge Z\} \\
& \quad =
\frac{
(a - b) (b^2 c^2 + a^2 (b + c)^2 + a b c (2 b + 3 c))
}
{
4 (a + b) (a + c) (b + c) (a b + b c + a c)
}
> 0 \\
\end{split}
\]
and
\[
\begin{split}
& \bP\{Y < X \wedge Z\} - \bP\{Z < X \wedge Y\} \\
& \quad =
\frac{
(b - c) (b^2 c^2 + 2 a b c (b + c) +
   a^2 (b^2 + 3 b c + c^2))
}
{
4 (a + b) (a + c) (b + c) (a b + b c + a c)
}
> 0, \\
\end{split}
\]
so
\[
\bP\{X < Y \wedge Z\}
>
\bP\{Y < X \wedge Z\}
>
\bP\{Z < X \wedge Y\}.
\]

\section*{Appendix C. Avoiding small-school bias}

We consider how one might correct for small-school bias in a model problem
involving standardized testing.

There are $n$ schools of different sizes.
The schools draw their students at random, independently, from
the same infinite population.
At the beginning of the school year, the scores students would get on the standardized
test are modeled as IID.
Attending school $i$ for the year increases the expected value of a student's
test score by $s_i$, $i=1, \ldots, n$.
Let $S_{ij}$ be the score of the $j$th student at school $i$ at the end of the year.
In this model, $\{ S_{ij}-s_i \}$ are IID.

We wish to award a ``best school'' prize to exactly one school,
based on student scores on the standardized test.
We want the scheme to be \emph{fair}, in that if 
$s_1 = s_2 = \cdots = s_n$, then all schools are equally likely to win.

We want the scheme to be \emph{valid} in the sense that if 
if $s_i > s_j$, then school $i$ is more likely to be picked as ``best 
school'' than $s_j$.

The proposed solution (suggested to us by Alex Rivest) is both 
fair and valid.

Let $m$ be the smallest school size.
The summary score for school $i$ is the average test score of
a random sample of $m$ students at school $i$.
The prize is awarded to the school with the highest summary score.

The method is fair, since the summary score for each school is determined
by a random size-$m$ set of students: If $\{s_i\}$ are equal,
the summary scores of the $n$ schools are IID, and every school is equally
likely to rank first. 
The method is valid, since the score of school $i$ is stochastically larger than
the score for school $j$ if $s_i > s_j$.

While this method is fair and valid, it relies on a subsample, so it 
might not maximize the probability that the prize is awarded to the
school with the largest $s_i$ among all far and valid methods.
Finding a better method that is both fair and valid is an open problem.

\end{document}